\documentclass[12pt]{article}
\usepackage{amsmath,amsthm,amssymb}
\usepackage{amssymb,latexsym}

\newtheorem{theorem}{Theorem}
\newtheorem{conjecture}{Conjecture}
\newtheorem{lemma}{Lemma}

\begin{document}
\baselineskip=17pt

\title{\bf Diophantine approximation with one prime of the form $\mathbf{p=x^2+y^2+1}$}

\author{\bf S. I. Dimitrov}

\date{\bf2020}
\maketitle

\begin{abstract}
Let $\varepsilon>0$ be a small constant.
In the present paper we prove that whenever $\eta$ is real and constants $\lambda _i$ satisfy some necessary
conditions, then there exist infinitely many prime triples $p_1,\, p_2,\, p_3$ satisfying the inequality
\begin{equation*}
|\lambda _1p_1 + \lambda _2p_2 + \lambda _3p_3+\eta|<\varepsilon
\end{equation*}
and such that $p_3=x^2 + y^2 +1$.\\
\quad\\
\textbf{Keywords}:  Diophantine approximation, primes.\\
\quad\\
{\bf  2010 Math.\ Subject Classification}: 11D75  $\cdot$  11P32
\end{abstract}

\section{Notations}
\indent

The letter $p$ with or without subscript will always denote prime numbers.
We denote by $(m,n)$ the greatest common divisor of $m$ and $n$. Moreover $e(t)$=exp($2\pi it$).
As usual $\varphi(d)$ is Euler's function, $r(d)$ is the number of solutions of the equation
$d=m_1^2+m_2^2$ in integers $m_j$, $\chi(d)$ is the non-principal character modulo 4
and $L(s,\chi)$ is the corresponding Dirichlet's $L$ -- function.
We shall use the convention that a congruence, $m\equiv n\,\pmod {d}$
will be written as $m\equiv n\,(d)$.
We denote by $\lfloor t\rfloor$, $\lceil t\rceil$ and $\{t\}$ respectively the  floor function,
the ceiling function and the fractional part function of $t$.
Let $\lambda_1,\lambda_2,\lambda_3$ are non-zero real numbers,
not all of the same sign and $\lambda_1/\lambda_2$ is irrational.
Then there are infinitely many different convergents
$a_0/q_0$ to its continued fraction, with
\begin{equation}\label{lambda12a0q0}
\bigg|\frac{\lambda_1}{\lambda_2} - \frac{a_0}{q_0}\bigg|<\frac{1}{q_0^2}\,,\quad (a_0, q_0) = 1\,,\quad a_0\neq0
\end{equation}
and $q_0$ is arbitrary large.
Denote
\begin{align}
\label{X}
&q_0^2=\frac{X}{(\log X)^{22}}\,;\\
\label{D}
&D=\frac{X^{1/2}}{(\log X)^{52}}\,;\\
\label{Delta}
&\Delta=\frac{(\log X)^{23}}{X}\,;\\
\label{theta0}
&\theta_0=\frac{1}{2}-\frac{1}{4}e\log2=0.0289...;\\
\label{varepsilon}
&\varepsilon=\frac{(\log\log X)^7}{(\log X)^{\theta_0}}\,;\\
\label{H}
&H=\frac{\log^2X}{\varepsilon}\,;\\
\label{SldalphaX}
&S_{l,d;J}(\alpha, X)=\sum\limits_{p\in J\atop{p\equiv l\, (d)}} e(\alpha p)\log p
\,,\quad J\subset(\lambda_0X,X]\,, \quad 0<\lambda_{0}<1\,;\\
\label{SalphaX}
&S(\alpha, X)=S_{1,1;(\lambda_0X,X]}(\alpha, X)\,;\\
\label{IJalphaX}
&I_J(\alpha,X)=\int\limits_Je(\alpha y)\,dy\,;\\
\label{IalphaX}
&I(\alpha,X)=I_{(\lambda_0X,X]}(\alpha, X)\,;\\
\label{Exqa}
&E(x,q,a)=\sum_{p\le x\atop{p\equiv a\,(q)}} \, \log p-\frac{x}{\varphi(q)}\,.
\end{align}

\section{Introduction and statement of the result}
\indent

In 1960 Linnik \cite{Linnik} has proved that there exist infinitely many prime numbers of the form
$p=x^2 + y^2 +1$, where $x$ and $y$ -- integers. More precisely he has proved the asymptotic formula
\begin{equation*}
\sum_{p\leq X}r(p-1)=\pi\prod_{p>2}\bigg(1+\frac{\chi(p)}{p(p-1)}\bigg)\frac{X}{\log X}+
\mathcal{O}\bigg(\frac{X(\log\log X)^7}{(\log X)^{1+\theta_0}}\bigg)\,,
\end{equation*}
where $\theta_0$ is defined by \eqref{theta0}.

Seven  years later Baker \cite{Baker} showed that whenever
$\lambda_1,\lambda_2,\lambda_3$ are non-zero real numbers, not all of
the same sign, $\lambda_1/\lambda_2$ is irrational and $\eta$ is real,
then there are infinitely many prime triples $p_1,\,p_2,\,p_3$ such that
\begin{equation}\label{Inequality1}
|\lambda_1p_1+\lambda_2p_2+\lambda_3p_3+\eta|<\xi\,,
\end{equation}
where $\xi=(\log \max p_j)^{-A}$ and $A>0$ is an arbitrary large constant.

Latter the right-hand side of \eqref{Inequality1} was sharpened several times
and the best result up to now belongs to K. Matom\"{a}ki \cite{Matomaki}
with $\xi=( \max p_j)^{-2/9+\delta}$ and $\delta>0$.

After Matom\"{a}ki inequality \eqref{Inequality1} was solved
with prime numbers of a special form.

Let $P_l$ is a number with at most $l$ prime factors.
The author and Todorova \cite{Dimitrov4}, and the author \cite{Dimitrov2}
proved that \eqref{Inequality1} has a solution in primes $p_i$ such that
$p_i+2=P_l$,\;$i=1,\,2,\,3$.

Very recently the author \cite{Dimitrov3} showed that \eqref{Inequality1}
has a solution  in Piatetski-Shapiro primes $p_1,\,p_2,\,p_3$
of type $\gamma\in(37/38, 1)$.

In this paper we continue to solve inequality \eqref{Inequality1}
with prime numbers of a special type.
More precisely we shall prove solvability of \eqref{Inequality1}
with Linnik primes. Thus we establish the following theorem.
\begin{theorem}\label{Theorem}
Suppose that $\lambda_1,\lambda_2,\lambda_3$ are non-zero real numbers, not all of
the same sign, $\lambda_1/\lambda_2$ is irrational and $\eta$ is real.
Then there exist infinitely many triples of primes $p_1,\,p_2,\,p_3$ for which
\begin{equation*}
|\lambda_1p_1+\lambda_2p_2+\lambda_3p_3+\eta|<\frac{(\log\log \max p_j)^7}{(\log \max p_j)^{\theta_0}}
\end{equation*}
and such that $p_3=x^2 + y^2 +1$. Here $\theta_0$ is defined by \eqref{theta0}.
\end{theorem}

\vspace{1mm}

In addition we have the following challenge.
\begin{conjecture}  Let $\varepsilon>0$ be a small constant.
Suppose that $\lambda_1,\lambda_2,\lambda_3$ are non-zero real numbers, not all of
the same sign, $\lambda_1/\lambda_2$ is irrational and $\eta$ is real.
Then there exist infinitely many triples of primes $p_1,\,p_2,\,p_3$ for which
\begin{equation*}
|\lambda_1p_1+\lambda_2p_2+\lambda_3p_3+\eta|<\varepsilon
\end{equation*}
and such that $p_1=x_1^2 + y_1^2 +1$,\,  $p_2=x_2^2 + y_2^2 +1$,\, $p_3=x_3^2 + y_3^2 +1$.
\end{conjecture}

The author wishes success to all young researchers in attacking this hard hypothesis.

\section{Preliminary lemmas}
\indent

\begin{lemma}\label{Fourier} Let $\varepsilon>0$ and $k\in \mathbb{N}$.
There exists a function $\theta(y)$ which is $k$ times continuously differentiable and
such that
\begin{align*}
&\theta(y)=1\quad\quad\quad\mbox{for }\quad\quad|y|\leq 3\varepsilon/4\,;\\
&0<\theta(y)<1\quad\mbox{for}\quad3\varepsilon/4 <|y|< \varepsilon\,;\\
&\theta(y)=0\quad\quad\quad\mbox{for}\quad\quad|y|\geq \varepsilon\,.
\end{align*}
and its Fourier transform
\begin{equation*}
\Theta(x)=\int\limits_{-\infty}^{\infty}\theta(y)e(-xy)dy
\end{equation*}
satisfies the inequality
\begin{equation*}
|\Theta(x)|\leq\min\bigg(\frac{7\varepsilon}{4},\frac{1}{\pi|x|},\frac{1}{\pi |x|}
\bigg(\frac{k}{2\pi |x|\varepsilon/8}\bigg)^k\bigg)\,.
\end{equation*}
\end{lemma}
\begin{proof}
See (\cite{Shapiro}).
\end{proof}

\begin{lemma}\label{SIasympt} Let $|\alpha|\leq\Delta$.
Then for the sum denoted by \eqref{SalphaX} and the integral denoted by \eqref{IalphaX}
the asymptotic formula
\begin{equation*}
S(\alpha,X)=I(\alpha,X)+\mathcal{O}\left(\frac{X}{e^{(\log X)^{1/5}}}\right)
\end{equation*}
holds.
\end{lemma}
\begin{proof}
Arguing as in (\cite{Tolev}, Lemma 14) we establish Lemma \ref{SIasympt}.
\end{proof}

\begin{lemma}\label{Bomb-Vin}(Bombieri -- Vinogradov) For any $C>0$ the following inequality
\begin{equation*}
    \sum\limits_{q\le X^{\frac{1}{2}}/(\log X)^{C+5}}
    \max\limits_{y\le X}\max\limits_{(a,\,q)=1}
    \big|E (y,\,q,\,a)\big|\ll
    \frac{X}{(\log X)^{C}}
\end{equation*}
holds.
\end{lemma}
\begin{proof}
See (\cite{Davenport}, Ch.28).
\end{proof}

\begin{lemma}\label{Expsumest}
Suppose that $\alpha \in \mathbb{R}$,\, $a \in \mathbb{Z}$,\, $q\in \mathbb{N}$,\,
$\big|\alpha-\frac{a}{q}\big|\leq\frac{1}{q^2}$\,, $(a, q)=1$.

Let
\begin{equation*}
\Sigma(\alpha,X)=\sum\limits_{p\le X}e(\alpha p)\log p\,.
\end{equation*}
Then
\begin{equation*}
\Sigma(\alpha, X)\ll \Big(Xq^{-1/2}+X^{4/5}+X^{1/2}q^{1/2}\Big)\log^4X\,.
\end{equation*}
\end{lemma}
\begin{proof}
See (\cite{Iwaniec-Kowalski}, Theorem 13.6).
\end{proof}

\begin{lemma}\label{Halberstam-Richert} Let $k\in\mathbb N$; $l,a,b\in\mathbb Z$ and $ab\neq0$.
Let $x$ and $y$ be real numbers satisfying
\begin{equation*}
k<y\leq x\,.
\end{equation*}
Then
\begin{equation*}
\#\{p\,:x-y<p\leq x,\,p\equiv l\,(k),\,ap+b=p'\}
\end{equation*}
\begin{equation*}
\ll\prod\limits_{p\mid kab}\bigg(1-\frac{1}{p}\bigg)^{-1}\frac{y}{\varphi(k)\log^2(y/k)}\,.
\end{equation*}
\end{lemma}
\begin{proof}
See (\cite{Halberstam}, Ch.2, Corollary 2.4.1).
\end{proof}
The next two lemmas are due to C. Hooley.
\begin{lemma}\label{Hooley1}
For any constant $\omega>0$ we have
\begin{equation*}
\sum\limits_{p\leq X}\bigg|\sum\limits_{d|p-1\atop{\sqrt{X}(\log X)^{-\omega}<d<\sqrt{X}(\log X)^{\omega}}}
\chi(d)\bigg|^2\ll \frac{X(\log\log X)^7}{\log X}\,,
\end{equation*}
where the constant in the Vinogradov symbol depends on $\omega>0$.
\end{lemma}

\begin{lemma}\label{Hooley2} Suppose that $\omega>0$ is a constant
and let $\mathcal{F}_\omega(X)$ be the number of primes $p\leq X$
such that $p-1$ has a divisor in the interval $\big(\sqrt{X}(\log X)^{-\omega}, \sqrt{X}(\log X)^\omega\big)$.
Then
\begin{equation*}
\mathcal{F}_\omega(X)\ll\frac{X(\log\log X)^3}{(\log X)^{1+2\theta_0}}\,,
\end{equation*}
where $\theta_0$ is defined by \eqref{theta0} and the constant in the Vinogradov symbol depends only on $\omega>0$.
\end{lemma}
The proofs of very similar results are available in (\cite{Hooley}, Ch.5).

\section{Outline of the proof}
\indent

Consider the sum
\begin{equation}\label{Gamma}
\Gamma(X)= \sum\limits_{\lambda_0X<p_1,p_2,p_3\leq X\atop{|\lambda_1p_1+\lambda_2p_2+\lambda_3p_3+\eta|<\varepsilon}}
r(p_3-1)\log p_1\log p_2\log p_3\,.
\end{equation}
Any non-trivial lower bound of $\Gamma(X)$ implies solvability of
$|\lambda_1p_1+\lambda_2p_2+\lambda_3p_3+\eta|<\varepsilon$ in primes such that $p_3=x^2 + y^2 +1$.

We have
\begin{equation}\label{GammaGamma0}
\Gamma(X)\geq\Gamma_0(X)\,,
\end{equation}
where
\begin{equation}\label{Gamma0}
\Gamma_0(X)=\sum\limits_{\lambda_0X<p_1,p_2,p_3\leq X}r(p_3-1)
\theta(\lambda_1p_1+\lambda_2p_2+\lambda_3p_3+\eta)\log p_1 \log p_2\log p_3\,.
\end{equation}
Using \eqref{Gamma0} and well-known identity $r(n)=4\sum_{d|n}\chi(d)$ we write
\begin{equation} \label{Gamma0decomp}
\Gamma_0(X)=4\big(\Gamma_1(X)+\Gamma_2(X)+\Gamma_3(X)\big),
\end{equation}
where
\begin{align}
\label{Gamma1}
&\Gamma_1(X)=\sum\limits_{\lambda_0X<p_1,p_2,p_3\leq X}
\left(\sum\limits_{d|p_3-1\atop{d\leq D}}\chi(d)\right)
\theta(\lambda_1p_1+\lambda_2p_2+\lambda_3p_3+\eta)\log p_1\log p_2\log p_3\,,\\
\label{Gamma2}
&\Gamma_2(X)=\sum\limits_{\lambda_0X<p_1,p_2,p_3\leq X}
\left(\sum\limits_{d|p_3-1\atop{D<d<X/D}}\chi(d)\right)
\theta(\lambda_1p_1+\lambda_2p_2+\lambda_3p_3+\eta)\log p_1\log p_2\log p_3\,,\\
\label{Gamma3}
&\Gamma_3(X)=\sum\limits_{\lambda_0X<p_1,p_2,p_3\leq X}
\left(\sum\limits_{d|p_3-1\atop{d\geq X/D}}\chi(d)\right)
\theta(\lambda_1p_1+\lambda_2p_2+\lambda_3p_3+\eta)\log p_1\log p_2\log p_3\,.
\end{align}
In order to estimate $\Gamma_1(X)$ and $\Gamma_3(X)$ we have to consider
the sum
\begin{equation} \label{Ild}
I_{l,d;J}(X)=\sum\limits_{\lambda_0X<p_1,p_2\leq X\atop{p_3\equiv l\,(d)
\atop{p_3\in J}}}\theta(\lambda_1p_1+\lambda_2p_2+\lambda_3p_3+\eta)\log p_1\log p_2\log p_3\,,
\end{equation}
where $d$ and $l$ are coprime natural numbers, and $J\subset(\lambda_0X,X]$-interval.
If $J=(\lambda_0X,X]$ then we write for simplicity $I_{l,d}(X)$.

Using the inverse Fourier transform for the function $\theta(x)$ we get
\begin{align*}
I_{l,d;J}(X)&=\sum\limits_{\lambda_0X<p_1,p_2\leq X\atop{p_3\equiv l\,(d)\atop{p_3\in J}}}\log p_1\log p_2\log p_3
\int\limits_{-\infty}^{\infty}\Theta(t)e\big((\lambda_1p_1+\lambda_2p_2+\lambda_3p_3+\eta)t\big)\,dt\\
&=\int\limits_{-\infty}^{\infty}\Theta(t)S(\lambda_1t,X)S(\lambda_2t,X)S_{l,d;J}(\lambda_3t,X)e(\eta t)\,dt\,.
\end{align*}
We decompose $I_{l,d;J}(X)$ as follows
\begin{equation}\label{Ilddecomp}
I_{l,d;J}(X)=I^{(1)}_{l,d;J}(X)+I^{(2)}_{l,d;J}(X)+I^{(3)}_{l,d;J}(X)\,,
\end{equation}
where
\begin{align}
\label{Ild1}
&I^{(1)}_{l,d;J}(X)=\int\limits_{|t|<\Delta}\Theta(t)S(\lambda_1t,X)S(\lambda_2t,X)S_{l,d;J}(\lambda_3t,X)e(\eta t)\,dt\,,\\
\label{Ild2}
&I^{(2)}_{l,d;J}(X)=\int\limits_{\Delta\leq|t|\leq H}\Theta(t)S(\lambda_1t,X)S(\lambda_2t,X)S_{l,d;J}(\lambda_3t,X)e(\eta t)\,dt\,,\\
\label{Ild3}
&I^{(3)}_{l,d;J}(X)=\int\limits_{|t|>H}\Theta(t)S(\lambda_1t,X)S(\lambda_2t,X)S_{l,d;J}(\lambda_3t,X)e(\eta t)\,dt\,.
\end{align}
We shall estimate $I^{(1)}_{l,d;J}(X)$, $I^{(3)}_{l,d;J}(X)$,
$\Gamma_3(X),\,\Gamma_2(X)$ and $\Gamma_1(X)$, respectively,
in the sections \ref{SectionIld1}, \ref{SectionIld3},
\ref{SectionGamma3}, \ref{SectionGamma2} and \ref{SectionGamma1}.
In section \ref{Sectionfinal} we shall complete the proof of Theorem \ref{Theorem}.

\section{Asymptotic formula for $\mathbf{I^{(1)}_{l,d;J}(X)}$}\label{SectionIld1}
\indent

Replace
\begin{align}
\label{S1}
&S_1=S(\lambda_1t, X)\,,\\
\label{S2}
&S_2=S(\lambda_2t, X)\,, \\
\label{S3}
&S_3=S_{l,d;J}(\lambda_3t,X)\,,\\
\label{I1}
&I_1=I(\lambda_1t, X)\,,\\
\label{I2}
&I_2=I(\lambda_2t, X)\,, \\
\label{I3}
&I_3=\frac{1}{\varphi(d)}I_J(\lambda_3t, X)\,.
\end{align}
We use the identity
\begin{equation}\label{Identity}
S_1S_2S_3=I_1I_2I_3+(S_1-I_1)I_2I_3+S_1(S_2-I_2)I_3+S_1S_2(S_3-I_3)\,.
\end{equation}

From \eqref{Delta}, \eqref{SldalphaX}, \eqref{IJalphaX}, \eqref{Exqa}, \eqref{S3}, \eqref{I3}
and Abel's summation formula it follows
\begin{equation}\label{S3I3}
S_3=I_3+\mathcal{O}\bigg(\Delta X\max\limits_{y\in(\lambda_{0}X,X]}\big|E(y, d, l)\big|\bigg)\,.
\end{equation}
Now using  \eqref{SalphaX} -- \eqref{IalphaX}, \eqref{S1} -- \eqref{S3I3}, Lemma \ref{SIasympt}
and the trivial estimations
\begin{equation*}
S_1, S_2, I_2\ll X \,, \quad I_3\ll \frac{X}{\varphi(d)}
\end{equation*}
we get
\begin{equation}\label{S123I123}
S_1S_2S_3-I_1I_2I_3\ll X^3\Bigg(\frac{1}{\varphi(d)e^{(\log X)^{1/5}}}
+\Delta\max\limits_{y\in(\lambda_{0}X,X]}\big|E(y,d,l)\big|\Bigg)\,.
\end{equation}
Put
\begin{equation}\label{PhiX}
\Phi(X)=\frac{1}{\varphi(d)}\int\limits_{|t|<\Delta}\Theta(t)I(\lambda_1t,X)I(\lambda_2t,X)I_J(\lambda_3t,X)e(\eta t)\,dt\,.
\end{equation}
Taking into account   \eqref{Ild1},   \eqref{S123I123}, \eqref{PhiX} and Lemma \ref{Fourier} we find
\begin{equation}\label{Ild1-PhiX}
I^{(1)}_{l,d;J}(X)-\Phi(X)\ll
\varepsilon\Delta X^3\Bigg(\frac{1}{\varphi(d)e^{(\log X)^{1/5}}}
+\Delta\max\limits_{y\in(\lambda_{0}X,X]}\big|E(y,d,l)\big|\Bigg)\,.
\end{equation}
On the other hand for the integral defined by \eqref{PhiX} we write
\begin{equation}\label{JXest1}
\Phi(X)=\frac{1}{\varphi(d)}B_J(X)+\Omega\,,
\end{equation}
where
\begin{equation*}
B_J(X)=\int\limits_J\int\limits_{\lambda_0X}^{X}\int\limits_{\lambda_0X}^{X}
\theta(\lambda_1y_1+\lambda_2y_2+\lambda_3y_3+\eta)\,dy_1\,dy_2\,dy_3
\end{equation*}
and
\begin{equation}\label{Omega}
\Omega\ll\frac{1}{\varphi(d)}\int\limits_{\Delta}^{\infty }|\Theta(t)|
|I(\lambda_1t,X)I(\lambda_2t,X)I_J(\lambda_3t,X)|\,dt\,.
\end{equation}
By  \eqref{IJalphaX} and \eqref{IalphaX} we get
\begin{equation}\label{IalphaXest}
I_J(\alpha,X)\ll\frac{1}{|\alpha|}\,, \quad  I(\alpha,X)\ll\frac{1}{|\alpha|}\,.
\end{equation}
Using \eqref{Omega}, \eqref{IalphaXest} and Lemma \ref{Fourier} we deduce
\begin{equation}\label{Omegaest}
\Omega\ll\frac{\varepsilon}{\varphi(d)\Delta^2}\,.
\end{equation}
Bearing in mind \eqref{Delta}, \eqref{Ild1-PhiX}, \eqref{JXest1} and \eqref{Omegaest} we find
\begin{equation}\label{Ild1est}
I^{(1)}_{l,d;J}(X)=\frac{1}{\varphi(d)}B_J(X)
+\varepsilon\Delta^2 X^3\max\limits_{y\in(\lambda_{0}X,X]}\big|E(y,d,l)\big|
+\frac{\varepsilon}{\varphi(d)\Delta^2}\,.
\end{equation}

\section{Upper bound  of $\mathbf{I^{(3)}_{l,d;J}(X)}$}\label{SectionIld3}
\indent

By \eqref{SldalphaX}, \eqref{SalphaX}, \eqref{Ild3} and Lemma \ref{Fourier} it follows
\begin{equation}\label{Ild3est1}
I^{(3)}_{l,d;J}(X)\ll \frac{X^3\log X}{d}\int\limits_{H}^{\infty}\frac{1}{t}\bigg(\frac{k}{2\pi t\varepsilon/8}\bigg)^k \,dt
=\frac{X^3\log X}{dk}\bigg(\frac{4k}{\pi\varepsilon H}\bigg)^k\,.
\end{equation}
Choosing $k=[\log X]$ from \eqref{H} and \eqref{Ild3est1} we obtain
\begin{equation}\label{Ild3est2}
I^{(3)}_{l,d;J}(X)\ll\frac{1}{d}\,.
\end{equation}

\section{Upper bound of $\mathbf{\Gamma_3(X)}$}\label{SectionGamma3}
\indent

Consider the sum  $\Gamma_3(X)$.\\
Since
\begin{equation*}
\sum\limits_{d|p_3-1\atop{d\geq X/D}}\chi(d)=\sum\limits_{m|p_3-1\atop{m\leq (p_3-1)D/X}}
\chi\bigg(\frac{p_3-1}{m}\bigg)
=\sum\limits_{j=\pm1}\chi(j)\sum\limits_{m|p_3-1\atop{m\leq (p_3-1)D/X
\atop{\frac{p_3-1}{m}\equiv j\;(\textmd{mod}\,4)}}}1
\end{equation*}
then from \eqref{Gamma3} and \eqref{Ild} it follows
\begin{equation*}
\Gamma_3(X)=\sum\limits_{m<D\atop{2|m}}\sum\limits_{j=\pm1}\chi(j)I_{1+jm,4m;J_m}(X)\,,
\end{equation*}
where $J_m=\big(\max\{1+mX/D,\lambda_0X\},X\big]$.
The last formula and \eqref{Ilddecomp} yield
\begin{equation}\label{Gamma3decomp}
\Gamma_3(X)=\Gamma_3^{(1)}(X)+\Gamma_3^{(2)}(X)+\Gamma_3^{(3)}(X)\,,
\end{equation}
where
\begin{equation}\label{Gamma3i}
\Gamma_3^{(i)}(X)=\sum\limits_{m<D\atop{2|m}}\sum\limits_{j=\pm1}\chi(j)
I_{1+jm,4m;J_m}^{(i)}(X)\,,\;\; i=1,\,2,\,3.
\end{equation}

\subsection{Estimation of $\mathbf{\Gamma_3^{(1)}(X)}$}
\indent

First we consider $\Gamma_3^{(1)}(X)$. From \eqref{Ild1est} and \eqref{Gamma3i} we deduce
\begin{align}\label{Gamma31}
\Gamma_3^{(1)}(X)=\Gamma^*&+\mathcal{O}\Big(\varepsilon\Delta^2 X^3\Sigma_1\Big)
+\mathcal{O}\bigg(\frac{\varepsilon}{\Delta^2}\Sigma_2\bigg)\,,
\end{align}
where
\begin{align}
\label{Gamma*}
&\Gamma^*=B_J(X)\sum\limits_{m<D\atop{2|m}}\frac{1}{\varphi(4m)}\sum\limits_{j=\pm1}\chi(j)\,,\\
\label{Sigma1}
&\Sigma_1=\sum\limits_{m<D\atop{2|m}}\max\limits_{y\in(\lambda_{0}X,X]}\big|E(y,4m,1+jm)\big|\,,\\
\label{Sigma2}
&\Sigma_2=\sum\limits_{m<D}\frac{1}{\varphi(4m)}\,.
\end{align}
From the properties of $\chi(k)$ we have that
\begin{equation}\label{Gamma*est}
\Gamma^*=0\,.
\end{equation}
By \eqref{D}, \eqref{Sigma1} and Lemma \ref{Bomb-Vin} we get
\begin{equation}\label{Sigma1est}
\Sigma_1\ll\frac{X}{(\log X)^{47}}\,.
\end{equation}
It is well known that
\begin{equation}\label{Sigma2est}
\Sigma_2\ll \log X\,.
\end{equation}
Bearing in mind \eqref{Delta}, \eqref{Gamma31}, \eqref{Gamma*est}, \eqref{Sigma1est} and \eqref{Sigma2est} we obtain
\begin{equation}\label{Gamma31est}
\Gamma_3^{(1)}(X)\ll\frac{\varepsilon X^2}{\log X}\,.
\end{equation}

\subsection{Estimation of $\mathbf{\Gamma_3^{(2)}(X)}$}
\indent

Next we consider $\Gamma_3^{(2)}(X)$. From \eqref{Ild2} and \eqref{Gamma3i} we have
\begin{equation}\label{Gamma32}
\Gamma_3^{(2)}(X)=\int\limits_{\Delta\leq|t|\leq H}\Theta(t)
S(\lambda_1t,X)S(\lambda_2t ,X)K(\lambda_3 t, X)e(\eta t)\,dt\,,
\end{equation}
where
\begin{equation}\label{Klambda3X}
K(\lambda_3t, X)=\sum\limits_{m<D\atop{2|m}}\sum\limits_{j=\pm1}\chi(j)S_{1+jm,4m;J_m}(\lambda_3t)\,.
\end{equation}
Suppose that
\begin{equation}\label{alphaaq}
\bigg|\alpha -\frac{a}{q}\bigg|\leq\frac{1}{q^2}\,,\quad (a, q)=1
\end{equation}
with
\begin{equation}\label{Intq}
q\in\left[(\log X)^{22},\,\frac{X}{(\log X)^{22}}\right]\,.
\end{equation}
Then \eqref{SalphaX}, \eqref{alphaaq}, \eqref{Intq} and  Lemma \ref{Expsumest} give us
\begin{equation}\label{Salphaest}
S(\alpha,\,X)\ll \frac{X}{(\log X)^7}\,.
\end{equation}
Let
\begin{equation}\label{mathfrakS}
\mathfrak{S}(t,X)=\min\left\{\left|S(\lambda_{1}t,\,X)\right|,\left|S(\lambda_2 t,\,X)\right|\right\}\,.
\end{equation}
Using  \eqref{lambda12a0q0}, \eqref{Salphaest}, \eqref{mathfrakS} and working similarly to  (\cite{Dimitrov4}, Lemma 6)
we establish that there exists a sequence of real numbers $X_1,\,X_2,\ldots \to \infty $ such that
\begin{equation}\label{mathfrakSest}
\mathfrak{S}(t, X_j)\ll \frac{X_j}{(\log X_j)^7}\,,\;\; j=1,2,\dots\,.
\end{equation}
Using \eqref{Gamma32}, \eqref{mathfrakS}, \eqref{mathfrakSest} and Lemma \ref{Fourier} we obtain
\begin{align}\label{Gamma32est1}
\Gamma_3^{(2)}(X_j)&\ll\varepsilon\int\limits_{\Delta\leq|t|\leq H}\mathfrak{S}(t, X_j)
\Big(\big|S(\lambda_1 t, X_j)K(\lambda_3 t, X_j)\big|
+\big|S(\lambda_2 t, X_j)K(\lambda_3 t, X_j)\big|\Big)\,dt\nonumber\\
&\ll\varepsilon\int\limits_{\Delta\leq|t|\leq H}\mathfrak{S}(t, X_j)
\Big(\big|S(\lambda_1 t, X_j)\big|^2
+\big|S(\lambda_2 t, X_j)\big|^2+\big|K(\lambda_3 t, X_j)\big|^2\Big)\,dt\nonumber\\
&\ll\varepsilon \frac{X_j}{(\log X_j)^7}\big(T_1+T_2+T_3\big)\,,
\end{align}
where
\begin{align}
\label{Tk}
&T_k=\int\limits_{\Delta}^H\big|S(\lambda_k t, X_j)\big|^2\,dt\,,\; k=1,2,\\
\label{T3}
&T_3=\int\limits_{\Delta}^H\big|K(\lambda_3 t, X_j)\big|^2\,dt\,.
\end{align}
From \eqref{Delta}, \eqref{H}, \eqref{SalphaX}, \eqref{Tk} and after straightforward computations  we get
\begin{equation}\label{Tkest}
T_k\ll HX_j\log X_j\,,\; k=1,2\,.
\end{equation}
Taking into account \eqref{Delta}, \eqref{H}, \eqref{Klambda3X}, \eqref{T3}
and proceeding as in (\cite{Dimitrov1},  p. 14) we find
\begin{equation}\label{T3est}
T_3\ll HX_j\log^3X_j\,.
\end{equation}
By \eqref{varepsilon}, \eqref{H}, \eqref{Tkest} and \eqref{T3est} we deduce
\begin{equation}\label{Gamma32est2}
\Gamma_3^{(2)}(X_j)\ll \frac{X_j}{(\log X_j)^7}X_j\log^5X_j
=\frac{X^2_j}{(\log X_j)^2}\ll\frac{\varepsilon X_j^2}{\log X_j}\,.
\end{equation}

\subsection{Estimation of $\mathbf{\Gamma_3^{(3)}(X)}$}
\indent

From \eqref{Ild3est2} and \eqref{Gamma3i} we have
\begin{equation}\label{Gamma33est}
\Gamma_3^{(3)}(X)\ll\sum\limits_{m<D}\frac{1}{d}\ll \log X\,.
\end{equation}

\subsection{Estimation of $\mathbf{\Gamma_3(X)}$}
\indent

Summarizing \eqref{Gamma3decomp}, \eqref{Gamma31est}, \eqref{Gamma32est2} and \eqref{Gamma33est} we get
\begin{equation}\label{Gamm3est}
\Gamma_3(X_j)\ll\frac{\varepsilon X_j^2}{\log X_j}\,.
\end{equation}

\section{Upper bound of $\mathbf{\Gamma_2(X)}$}\label{SectionGamma2}
\indent

Consider the sum $\Gamma_2(X)$. We denote by $\mathcal{F}(X)$ the set of all primes
$\lambda_0X<p\leq X$ such that $p-1$ has a divisor belongs to the interval $(D,X/D)$.
The inequality $xy\leq x^2+y^2$ and  \eqref{Gamma2} yield
\begin{align*}
\Gamma_2(X)^2&\ll(\log X)^6\sum\limits_{\lambda_0X<p_1,...,p_6\leq X
\atop{|\lambda_1p_1+\lambda_2p_2+\lambda_3p_3+\eta|<\varepsilon
\atop{|\lambda_1p_4+\lambda_2p_5+\lambda_3p_6+\eta|<\varepsilon}}}
\left|\sum\limits_{d|p_3-1\atop{D<d<X/D}}\chi(d)\right|
\left|\sum\limits_{t|p_6-1\atop{D<t<X/D}}\chi(t)\right|\\
&\ll(\log X)^6\sum\limits_{\lambda_0X<p_1,...,p_6\leq X
\atop{|\lambda_1p_1+\lambda_2p_2+\lambda_3p_3+\eta|<\varepsilon
\atop{\lambda_1p_4+\lambda_2p_5+\lambda_3p_6+\eta|<\varepsilon
\atop{p_6\in\mathcal{F}(X)}}}}\left|\sum\limits_{d|p_3-1
\atop{D<d<X/D}}\chi(d)\right|^2\,.
\end{align*}
The summands in the last sum for which $p_3=p_6$ can be estimated with
$\mathcal{O}\big(X^{3+\varepsilon}\big)$.\\
Therefore
\begin{equation}\label{Gamma2est1}
\Gamma_2(X)^2\ll(\log X)^6\Sigma_0+X^{3+\varepsilon}\,,
\end{equation}
where
\begin{equation}\label{Sigma0}
\Sigma_0=\sum\limits_{\lambda_0X<p_3\leq X}\left|\sum\limits_{d|p_3-1
\atop{D<d<X/D}}\chi(d)\right|^2\sum\limits_{\lambda_0X<p_6\leq X\atop{p_6\in\mathcal{F}
\atop{p_6\neq p_1}}}\sum\limits_{\lambda_0X<p_1,p_2,p_4,p_5\leq X
\atop{|\lambda_1p_1+\lambda_2p_2+\lambda_3p_3+\eta|<\varepsilon
\atop{|\lambda_1p_4+\lambda_2p_5+\lambda_3p_6+\eta|<\varepsilon}}}1\,.
\end{equation}
Since $\lambda_1,\lambda_2,\lambda_3$ are not all of the same sign,
without loss of generality we can assume that
$\lambda_1>0,\,\lambda_2>0$ and $\lambda_3<0$.
Now let us consider the set
\begin{equation}\label{SetPsi}
\Psi(X)=\{\langle p_1,p_2\rangle\;:\; |\lambda_1p_1+\lambda_2p_2+D|<\varepsilon,
\;\;\lambda_0X<p_1,p_2\leq X,\;\;D\asymp X\}\,.
\end{equation}
We shall find the upper bound of the cardinality of $\Psi(X)$.

Using
\begin{equation}\label{binaryinequality}
|\lambda_1p_1+\lambda_2p_2+D|<\varepsilon
\end{equation}
we write
\begin{equation}\label{p1p2D}
\left|\frac{\lambda_1}{\lambda_2}p_1+p_2+\frac{D}{\lambda_2}\right|<\frac{\varepsilon}{\lambda_2}\,.
\end{equation}
Since $\lambda_2$ is fixed then for sufficiently large $X$ we have that $\varepsilon/\lambda_2$
is sufficiently small. Therefore \eqref{p1p2D} implies

\textbf{Case 1.}
\begin{equation*}
\left\lfloor\frac{\lambda_1}{\lambda_2}p_1+p_2\right\rfloor=\left\lfloor-\frac{D}{\lambda_2}\right\rfloor
\end{equation*}
or

\textbf{Case 2.}
\begin{equation*}
\left\lceil\frac{\lambda_1}{\lambda_2}p_1+p_2\right\rceil=\left\lceil-\frac{D}{\lambda_2}\right\rceil
\end{equation*}
or

\textbf{Case 3.}

\begin{equation*}
\left\lfloor\frac{\lambda_1}{\lambda_2}p_1+p_2\right\rfloor=\left\lceil-\frac{D}{\lambda_2}\right\rceil
\end{equation*}
or

\textbf{Case 4.}

\begin{equation*}
\left\lceil\frac{\lambda_1}{\lambda_2}p_1+p_2\right\rceil=\left\lfloor-\frac{D}{\lambda_2}\right\rfloor\,.
\end{equation*}
We shall consider only the Case 1. The   Cases 2, 3 and 4 are treated similarly.

From Case 1 we have
\begin{equation*}
\left\lfloor\frac{\lambda_1}{\lambda_2}p_1\right\rfloor+p_2=\left\lfloor-\frac{D}{\lambda_2}\right\rfloor
\end{equation*}
thus
\begin{equation*}
\left\lfloor\left(\left\lfloor\frac{\lambda_1}{\lambda_2}\right\rfloor
+\left\{\frac{\lambda_1}{\lambda_2}\right\}\right)p_1\right\rfloor
+p_2=\left\lfloor-\frac{D}{\lambda_2}\right\rfloor
\end{equation*}
and therefore
\begin{equation}\label{Equality1}
\left\lfloor\frac{\lambda_1}{\lambda_2}\right\rfloor p_1
+p_2=\left\lfloor-\frac{D}{\lambda_2}\right\rfloor
-\left\lfloor\left\{\frac{\lambda_1}{\lambda_2}  p_1 \right\}\right\rfloor
\end{equation}
Bearing in mind the definition \eqref{SetPsi} we deduce that there exist
constants $C_1>0$ and $C_2>0$ such that
\begin{equation*}
C_1X\leq\left\lfloor-\frac{D}{\lambda_2}\right\rfloor
-\left\lfloor\left\{\frac{\lambda_1}{\lambda_2}  p_1 \right\}\right\rfloor\leq C_2X\,.
\end{equation*}
Consequently there exists constant $C\in[C_1, C_2]$ such that
\begin{equation}\label{Equality2}
\left\lfloor-\frac{D}{\lambda_2}\right\rfloor
-\left\lfloor\left\{\frac{\lambda_1}{\lambda_2}  p_1 \right\}\right\rfloor= CX\,.
\end{equation}
The equalities \eqref{Equality1} and \eqref{Equality2} give us
\begin{equation}\label{Equations}
\left\lfloor\frac{\lambda_1}{\lambda_2}\right\rfloor p_1+p_2=CX\,,
\end{equation}
for some constant $C\in[C_1, C_2]$.

We established that the number of solutions of the inequality \eqref{binaryinequality}
is less than the number of all solutions of all equations denoted by \eqref{Equations}.
According to Lemma \ref{Halberstam-Richert} for any fixed $C\in[C_1, C_2]$
participating in \eqref{Equations} we have
\begin{equation}\label{p1p2CX}
\#\{\langle p_1,p_2\rangle\;:\; \left\lfloor\lambda_1/\lambda_2\right\rfloor p_1+p_2=CX,
\;\;\lambda_0X<p_1,p_2\leq X\}\ll\frac{X\log\log X}{\log^2X}\,.
\end{equation}
Taking into account that $C\leq C_2$ from \eqref {SetPsi}, \eqref{binaryinequality} and \eqref{p1p2CX} we find
\begin{equation}\label{Psiest}
\#\Psi(X)\ll\frac{X\log\log X}{\log^2X}\,.
\end{equation}
The estimations \eqref{Sigma0} and \eqref{Psiest}  yield
\begin{equation}\label{Sigma0est}
\Sigma_0\ll\frac{X^2}{\log^4X}(\log\log X)^2\Sigma^\prime\Sigma^{\prime\prime}\,,
\end{equation}
where
\begin{equation*}
\Sigma^\prime=\sum\limits_{\lambda_0X<p\leq X}\left|\sum\limits_{d|p-1
\atop{D<d<X/D}}\chi(d)\right|^2\,,\quad \Sigma^{\prime\prime}=
\sum\limits_{\lambda_0X<p\leq X\atop{p\in\mathcal{F}}}1\,.
\end{equation*}
Applying Lemma \ref{Hooley1} we obtain
\begin{equation}\label{Sigma'est}
\Sigma^\prime\ll\frac{X(\log\log X)^7}{\log X}\,.
\end{equation}
Using Lemma \ref{Hooley2} we get
\begin{equation}\label{Sigma''est}
\Sigma^{\prime\prime}\ll\frac{X(\log\log X)^3}{(\log X)^{1+2\theta_0}}\,,
\end{equation}
where $\theta_0$ is denoted by  \eqref{theta0}.

We are now in a good position to estimate the sum $\Gamma_2(X)$.
From \eqref{Gamma2est1}, \eqref{Sigma0est} -- \eqref{Sigma''est} it follows
\begin{equation}\label{Gamma2est2}
\Gamma_2(X)\ll\frac{ X^2(\log\log X)^6}{(\log X)^{\theta_0}}=\frac{\varepsilon X^2}{\log\log X}\,.
\end{equation}

\section{Lower bound  for $\mathbf{\Gamma_1(X)}$}\label{SectionGamma1}
\indent

Consider the sum $\Gamma_1(X)$.
From \eqref{Gamma1}, \eqref{Ild} and \eqref{Ilddecomp} we deduce
\begin{equation}\label{Gamma1decomp}
\Gamma_1(X)=\Gamma_1^{(1)}(X)+\Gamma_1^{(2)}(X)+\Gamma_1^{(3)}(X)\,,
\end{equation}
where
\begin{equation}\label{Gamma1i}
\Gamma_1^{(i)}(X)=\sum\limits_{d\leq D}\chi(d)I_{1,d}^{(i)}(X)\,,\;\; i=1,\,2,\,3.
\end{equation}

\subsection{Estimation of $\mathbf{\Gamma_1^{(1)}(X)}$}
\indent

First we consider $\Gamma_1^{(1)}(X)$.
Using formula \eqref{Ild1est} for $J=(\lambda_0X,X]$, \eqref{Gamma1i}
and treating the reminder term by the same way as for $\Gamma_3^{(1)}(X)$
we find
\begin{equation} \label{Gamma11est1}
\Gamma_1^{(1)}(X)=B(X)\sum\limits_{d\leq D}\frac{\chi(d)}{\varphi(d)}
+\mathcal{O}\bigg(\frac{\varepsilon X^2}{\log X}\bigg)\,,
\end{equation}
where
\begin{equation*}
B(X)=\int\limits_{\lambda_0X}^{X}\int\limits_{\lambda_0X}^{X}\int\limits_{\lambda_0X}^{X}
\theta(\lambda_1y_1+\lambda_2y_2+\lambda_3y_3+\eta)\,dy_1\,dy_2\,dy_3\,.
\end{equation*}
According to (\cite{Dimitrov4}, Lemma 4) we have
\begin{equation}\label{Best}
B(X)\gg\varepsilon X^2\,.
\end{equation}
Denote
\begin{equation} \label{Sigmaf}
\Sigma=\sum\limits_{d\leq D}f(d)\,,\quad f(d)=\frac{\chi(d)}{\varphi(d)}\,.
\end{equation}
We have
\begin{equation}\label{fdest}
f(d)\ll d^{-1}\log\log(10d)
\end{equation}
with absolute constant in the Vinogradov's symbol. Hence the corresponding Dirichlet series
\begin{equation*}
F(s)=\sum\limits_{d=1}^\infty\frac{f(d)}{d^s}
\end{equation*}
is absolutely convergent in $Re(s)>0$. On the other hand $f(d)$ is a multiplicative with
respect to $d$ and applying Euler's identity we obtain
\begin{equation}\label{FT}
F(s)=\prod\limits_pT(p,s)\,,\quad T(p,s)=1+\sum\limits_{l=1}^\infty f(p^l)p^{-ls}\,.
\end{equation}
By \eqref{Sigmaf} and \eqref{FT} we establish that
\begin{equation*}
T(p,s)=\left(1-\frac{\chi(p)}{p^{s+1}}\right)^{-1}\left(1+\frac{\chi(p)}{p^{s+1}(p-1)}\right)\,.
\end{equation*}
Hence we find
\begin{equation}\label{Fs}
F(s)=L(s+1,\chi)\mathcal{N}(s)\,,
\end{equation}
where $L(s+1,\chi)$ -- Dirichlet series corresponding to the character $\chi$ and
\begin{equation}\label{Ns}
\mathcal{N}(s)=\prod\limits_p \left(1+\frac{\chi(p)}{p^{s+1}(p-1)}\right)\,.
\end{equation}
From the properties of the $L$ -- functions it follows that $F(s)$ has an analytic continuation to $Re(s)>-1$.
It is well known that
\begin{equation}\label{Lsest}
L(s+1,\chi)\ll1+\left|Im(s)\right|^{1/6}\quad \mbox{for}\quad Re(s)\geq-\frac{1}{2}\,.
\end{equation}
Moreover
\begin{equation}\label{Nsest}
\mathcal{N}(s)\ll1\,.
\end{equation}
By \eqref{Fs}, \eqref{Lsest} and \eqref{Nsest} we deduce
\begin{equation}\label{Fsest}
F(s)\ll X^{1/6}\quad \mbox{for}\quad Re(s)\geq-\frac{1}{2}\,,\quad |Im(s)|\leq X\,.
\end{equation}
Usin \eqref{Sigmaf}, \eqref{fdest} and Perron's formula given at Tenenbaum (\cite{Tenenbaum}, Chapter II.2)
we obtain
\begin{equation}\label{SigmaPeron}
\Sigma=\frac{1}{2\pi i}\int\limits_{\varkappa- iX}^{\varkappa+iX}F(s)\frac{D^s}{s}ds
+\mathcal{O}\left(\sum\limits_{t=1}^\infty\frac{D^\varkappa\log\log(10t)}{t^{1+\varkappa}
\left(1+X\left|\log\frac{D}{t}\right|\right)}\right)\,,
\end{equation}
where $\varkappa=1/10$. It is easy to see that the error term above is $\mathcal{O}\Big(X^{-1/20}\Big)$.
Applying the residue theorem we see that the main term in \eqref{SigmaPeron} is equal to
\begin{equation*}
F(0)+\frac{1}{2\pi i}\left(\int\limits_{1/10-iX}^{-1/2-i X}+
\int\limits_{-1/2-iX}^{-1/2+i X}+\int\limits_{-1/2+i X}^{1/10+i X}\right)F(s)\frac{D^s}{s}ds\,.
\end{equation*}
From \eqref{Fsest} it follows that the contribution from the above integrals is $\mathcal{O}\Big(X^{-1/20}\Big)$.\\
Hence
\begin{equation}\label{Sigmaest}
\Sigma=F(0)+\mathcal{O}\Big(X^{-1/20}\Big)\,.
\end{equation}
Using \eqref{Fs} we get
\begin{equation}\label{F0}
F(0)=\frac{\pi}{4}\mathcal{N}(0)\,.
\end{equation}
Bearing in mind \eqref{Gamma11est1}, \eqref{Sigmaf}, \eqref{Ns}, \eqref{Sigmaest}
and (\ref{F0}) we find a new expression for $\Gamma_1^{(1)}(X)$
\begin{equation}\label{Gamma11est2}
\Gamma_1^{(1)}(X)=\frac{\pi}{4}\prod\limits_p \left(1+\frac{\chi(p)}{p(p-1)}\right) B(X)
+\mathcal{O}\bigg(\frac{\varepsilon X^2}{\log X}\bigg)+\mathcal{O}\Big(B(X)X^{-1/20}\Big)\,.
\end{equation}
Now \eqref{Best} and \eqref{Gamma11est2} yield
\begin{equation}\label{Gamma11est3}
\Gamma_1^{(1)}(X)\gg\varepsilon X^2\,.
\end{equation}

\subsection{Estimation of $\mathbf{\Gamma_1^{(2)}(X)}$}
\indent

Arguing as in the estimation of $\Gamma_3^{(2)}(X)$ we get
\begin{equation} \label{Gamma12est}
\Gamma_1^{(2)}(X)\ll\frac{\varepsilon X^2}{\log X}\,.
\end{equation}

\subsection{Estimation of $\mathbf{\Gamma_1^{(3)}(X)}$}
\indent

From \eqref{Ild3est2} and \eqref{Gamma1i} we have
\begin{equation}\label{Gamma13est}
\Gamma_1^{(3)}(X)\ll\sum\limits_{m<D}\frac{1}{d}\ll \log X\,.
\end{equation}

\subsection{Estimation of $\mathbf{\Gamma_1(X)}$}
\indent

Summarizing  \eqref{Gamma1decomp}, \eqref{Gamma11est3}, \eqref{Gamma12est} and \eqref{Gamma13est} we deduce
\begin{equation} \label{Gamma1est}
\Gamma_1(X)\gg\varepsilon X^2\,.
\end{equation}

\section{Proof of the Theorem}\label{Sectionfinal}
\indent

Taking into account \eqref{varepsilon}, \eqref{GammaGamma0}, \eqref{Gamma0decomp},
\eqref{Gamm3est}, \eqref{Gamma2est2} and \eqref{Gamma1est} we obtain
\begin{equation*}
\Gamma(X_j)\gg\varepsilon X_j^2=\frac{X_j^2(\log\log X_j)^7}{(\log X_j)^{\theta_0}}\,.
\end{equation*}
The last lower bound implies
\begin{equation}\label{Lowerbound}
\Gamma(X_j) \rightarrow\infty \quad \mbox{ as } \quad X_j\rightarrow\infty\,.
\end{equation}
Bearing in mind  \eqref{Gamma} and \eqref{Lowerbound} we establish Theorem \ref{Theorem}.

\vskip20pt
\footnotesize
\begin{flushleft}
S. I. Dimitrov\\
Faculty of Applied Mathematics and Informatics\\
Technical University of Sofia \\
8, St.Kliment Ohridski Blvd. \\
1756 Sofia, BULGARIA\\
e-mail: sdimitrov@tu-sofia.bg\\
\end{flushleft}


\begin{thebibliography}{0}

\bibitem{Baker} A. Baker, {\it On some Diophantine inequalities involving primes},
J. Reine Angew. Math., {\bf 228}, (1967), 166 -- 181.

\bibitem{Davenport} H. Davenport, {\it Multiplicative number theory} (revised by H. Montgomery),
Springer, (2000), Third ed.

\bibitem{Dimitrov1} S. I. Dimitrov, {\it The ternary Goldbach problem with prime numbers of a mixed type},
Notes on Number Theory and Discrete Mathematics, \textbf{24}, 2, (2018), 6 -- 20.

\bibitem{Dimitrov2} S. I. Dimitrov, {\it  Diophantine approximation by special primes},
Appl. Math. in Eng. and Econ. -- 44th. Int. Conf., AIP Conf. Proc., \textbf{2048}, 050005, (2018).

\bibitem{Dimitrov3} S. I. Dimitrov, {\it  Diophantine approximation by Piatetski-Shapiro primes},
arXiv:2006.01003v1  [math.NT]  1 Jun 2020.

\bibitem{Dimitrov4} S. Dimitrov, T. Todorova, {\it Diophantine approximation by prime numbers of a special form},
Annuaire Univ. Sofia, Fac. Math. Inform., {\bf102}, (2015), 71 -- 90.

\bibitem{Halberstam} H. Halberstam, H.-E. Richert, {\it Sieve Methods},
Academic Press, (1974).

\bibitem{Hooley} C. Hooley, {\it Applications of sieve methods to the theory of numbers},
Cambridge Univ. Press, (1976).

\bibitem{Iwaniec-Kowalski} H. Iwaniec, E. Kowalski, {\it Analytic number theory},
Colloquium Publications, \textbf{53}, Amer. Math. Soc., (2004).

\bibitem{Linnik} Ju. Linnik, {\it An asymptotic formula in an additive problem of Hardy and Littlewood},
Izv. Akad. Nauk SSSR, Ser.Mat., {\bf24}, (1960), 629 -- 706 (in Russian).

\bibitem{Matomaki} K. Matom\"{a}ki, {\it Diophantine approximation by primes},
Glasgow Math. J., {\bf 52}, (2010), 87 -- 106.

\bibitem{Tenenbaum}G. Tenenbaum, {\it Introduction to Analytic and Probabilistic Number Theory},
Cambridge Univ. Press, (1995).

\bibitem{Tolev} D. Tolev, {\it On a diophantine inequality involving prime numbers},
Acta Arith., {\bf61}, (1992), 289 -- 306.

\bibitem{Shapiro} I. Piatetski-Shapiro, {\it On a variant of the Waring-Goldbach problem},
Mat. Sb., {\bf30}, (1952), 105 -- 120, (in Russian).

\end{thebibliography}
\end{document}